\documentclass[11pt]{article}

\usepackage[margin=1in]{geometry}
\usepackage[hidelinks]{hyperref}
\usepackage{microtype}
\usepackage{xcolor}
\usepackage{comment}

\usepackage{amsmath,amssymb,amsthm,mathtools}
\usepackage{enumitem}

\usepackage[nameinlink,noabbrev]{cleveref}

\theoremstyle{plain}
\newtheorem{theorem}{Theorem}[section]
\newtheorem{proposition}[theorem]{Proposition}
\newtheorem{corollary}[theorem]{Corollary}
\newtheorem{lemma}[theorem]{Lemma}

\theoremstyle{definition}
\newtheorem{definition}[theorem]{Definition}

\theoremstyle{remark}
\newtheorem{remark}[theorem]{Remark}

\theoremstyle{definition}
\newtheorem{example}[theorem]{Example}

\newcommand{\Q}{\mathbb{Q}}

\newcommand{\Catnum}[1]{\mathrm{Cat}_{#1}}
\newcommand{\Coeff}{\operatorname{Coeff}}

\newcommand{\Z}{\mathbb{Z}}
\newcommand{\C}{\mathbb{C}}
\newcommand{\Gal}{\mathrm{Gal}}

\newcommand{\ii}{\mathrm{i}}

\newcommand{\binompar}[2]{\binom{#1}{#2}_{\mathrm{par}}}

\title{Inverse-Limit Formulas and Stable-Range Rigidity for Cyclotomic Sums }

\author{Juan D. Vélez, Carlos Cadavid}
\date{\today}

\begin{document}
\maketitle

\begin{abstract}
We study truncation-compatible families \(F=(F_m)_{m\ge 1}\) over \(\Q[z]\) through an inverse-limit formalism, and we evaluate them at the punctured cyclotomic cosine points \(\alpha_{k,n}=\cos(2\pi k/n)\) with the specialization \(z=n-1\). For symmetric families of uniformly bounded total \(x\)-degree \(\le d\), we prove a stable-range rigidity theorem: for all \(n\ge d+2\), the cosine-point evaluation factors through the finitely many punctured cosine power sums \(P_1(n),\dots,P_d(n)\). In the purely polynomial case this implies eventual polynomiality in \(n\). We then extend the framework to include fixed product factors and package their cosine-point contribution in multiplicative invariants \(M_Q(n)\). In the stable range, the bounded-degree symmetric part collapses as before; any remaining cyclotomic dependence occurs only through these explicit product terms. Finally, we show that coefficient extraction from such products produces further bounded-degree symmetric families, and we apply this to complete symmetric functions \(h_r\) evaluated at cosine points.
\end{abstract}

\subsubsection*{Summary of results.}

Much of the classical literature on cyclotomic trigonometric sums is devoted to explicit evaluations at \emph{fixed levels}. In this paper we take a different point of view. Instead of treating each level \(n\) in isolation, we ask how \(n\)-dependent expressions behave once their algebraic complexity is fixed and the cyclotomic level becomes large.

Our basic objects are \emph{polynomial formulas}: truncation-compatible families \(F=(F_m)_{m\ge 1}\) with
\[
F_m\in\Q[z][x_1,\dots,x_m],
\qquad
F_{m+1}(x_1,\dots,x_m,0)=F_m(x_1,\dots,x_m).
\]
At level \(n\ge 2\) we evaluate such a family at the punctured cosine points
\[
\alpha_{k,n}=\cos\!\Bigl(\frac{2\pi k}{n}\Bigr),
\qquad 1\le k\le n-1,
\]
using \(m=n-1\) variables together with the specialization \(z=n-1\). This produces a numerical sequence
\[
a(n)=F_{n-1}(\alpha_{1,n},\dots,\alpha_{n-1,n})\Big|_{z=n-1}.
\]

\smallskip

Two types of universal invariants control these evaluations. The additive invariants are the punctured cosine power sums
\[
P_h(n)=\sum_{k=1}^{n-1}(2\alpha_{k,n})^h,
\]
while fixed unit-normalized product factors \(Q(z,t)\in\Q[z][t]\) with \(Q(z,0)=1\) give rise to multiplicative invariants
\[
M_Q(n)=\prod_{k=1}^{n-1}Q(n-1,\alpha_{k,n}).
\]

\bigskip

\subsubsection*{Main results.}

\begin{itemize}
\item \textbf{Inverse-limit formalism.}
We show that truncation-compatible families provide a natural language for \(n\)-dependent summation problems. Passing to an inverse limit packages the levelwise data into a single algebraic object and allows cosine-point evaluation to be formulated uniformly in \(n\).

\item \textbf{Stable-range rigidity.}
If a family is symmetric and has uniformly bounded total \(x\)-degree \(\le d\), then its cosine-point evaluation becomes rigid once \(n\) is large. More precisely, for all \(n\ge d+2\) the evaluation necessarily factors through the finite list of invariants \(P_1(n),\dots,P_d(n)\). Any remaining cyclotomic dependence is forced to appear only through explicit multiplicative factors \(M_Q(n)\).

\item \textbf{Effective evaluation of the invariants.}
Using binomial identities arising from discrete heat-kernel methods, we obtain explicit stable-range formulas for the punctured power sums \(P_h(n)\). This makes the rigidity theorem computationally effective.

\item \textbf{Eventual polynomiality.}
In the purely polynomial case, every admissible bounded-degree symmetric family has an eventual polynomial evaluation: there exists \(R_\infty(n)\in\Q[n]\) such that
\[
a(n)=R_\infty(n)\qquad(n\ge d+2).
\]
As a consequence, global identities reduce to checking finitely many small values of \(n\).

\item \textbf{Coefficient extraction as a construction principle.}
We show that extracting coefficients from unit-normalized products
\[
\prod_{j=1}^{m} Q(z,sx_j)
\]
systematically produces admissible bounded-degree symmetric families. This turns the theory from a purely reductive statement into a method for constructing new examples.

\item \textbf{Application to complete symmetric functions.}
Applying this mechanism to \(Q(t)=(1-t)^{-1}\) yields the complete symmetric polynomials \(h_r\). For each fixed \(r\), we compute explicit stable-range formulas for \(h_r(\alpha_{1,n},\dots,\alpha_{n-1,n})\), revealing an underlying Catalan--hypergeometric structure.
\end{itemize}

\subsubsection*{What is new.}

The contribution of this paper is mainly structural. The inverse-limit formalism makes precise the idea that an \(n\)-indexed construction should be treated as a single algebraic object, rather than as a collection of unrelated fixed-level expressions.

Within this framework, bounded-degree symmetry forces stabilization: once the level is large enough, only finitely many universal invariants can influence the evaluation. Eventual polynomiality and finite verification then follow as formal consequences, rather than as ad hoc computational phenomena. The fixed-product and coefficient-extraction extensions further clarify where genuine cyclotomic variation may persist and provide a systematic source of new admissible families.

\bigskip

\section{Introduction}
\label{sec:introduction}

Finite sums of trigonometric values at rational angles, and their higher power analogues, occur throughout analytic number theory, spectral theory, and related areas. A substantial literature is devoted to their explicit evaluation. For cosine power sums one may consult, for example, da~Fonseca and collaborators \cite{dFK13,dFGK17}, Merca \cite{Me12}, the heat-kernel approach of \cite{CHJSV23,JKS23}, and interpolation or partial-fraction techniques \cite{CM99,AH18,Ha08}.

A basic role in this paper is played by the power sums
\[
C(n,h)=\sum_{k=0}^{n-1}\cos^{h}\!\Bigl(\frac{2\pi k}{n}\Bigr),
\qquad
S(n,h)=\sum_{k=0}^{n-1}\sin^{h}\!\Bigl(\frac{2\pi k}{n}\Bigr),
\]
which serve as building blocks for more complicated cyclotomic expressions.

\medskip

The focus of the present work, however, is not the computation of individual sums at a fixed \(n\). Instead, we study a recurring phenomenon: once the algebraic complexity of an \(n\)-dependent expression is fixed, its dependence on the cyclotomic level \(n\) often stabilizes for all sufficiently large \(n\).

Our evaluations take place on the \emph{punctured} cyclotomic cosine points
\[
\alpha_{k,n}=\cos\!\Bigl(\frac{2\pi k}{n}\Bigr),
\qquad 1\le k\le n-1,
\]
where the trivial point \(k=0\) is omitted. At each \emph{level} \(n\ge 2\) we evaluate a formula in \(m=n-1\) variables on the tuple \((\alpha_{1,n},\dots,\alpha_{n-1,n})\), with an auxiliary parameter specialized to \(z=n-1\).

\medskip

A simple example already illustrates the phenomenon. The quadratic energy
\[
\sum_{1\le i<j\le m}(x_i-x_j)^2
\]
is a symmetric polynomial of degree \(2\), and its evaluation at the punctured cosine configuration simplifies to the closed form \(n(n-3)/2\). Our goal is to explain why such stabilization is not accidental, but forced by general algebraic principles.

\medskip

\noindent\emph{Guiding question.}
When an \(n\)-dependent expression is built symmetrically from the punctured cosine points, what mechanism forces its evaluation to stabilize as \(n\) grows?

\medskip

Two structural ideas lie at the heart of the answer.

\smallskip
\noindent\textbf{(i) Inverse-limit packaging.}
At level \(n\) the ambient polynomial ring depends on \(n\), since the number of variables is \(n-1\). Many natural constructions are compatible with forgetting the last variable, implemented algebraically by the substitution \(x_{m+1}=0\). Encoding this compatibility leads naturally to an inverse-limit formalism, which allows level-dependent expressions to be treated as a single algebraic object.

\smallskip
\noindent\textbf{(ii) Bounded-degree symmetry.}
If the total \(x\)-degree is uniformly bounded and the expression is symmetric, then once the number of variables is large compared to the degree, the polynomial can depend only on finitely many power sums. After specialization to cosine points, these become the punctured invariants \(P_1(n),\dots,P_d(n)\). Our main theorem shows that this finite dependence persists after evaluation: for all sufficiently large \(n\), no further cyclotomic information can appear.

\medskip

To make this reduction effective, we use binomial identities arising from discrete heat-kernel methods, which yield stable-range formulas for the invariants \(P_h(n)\). Combining these identities with the structural reduction leads to eventual polynomiality and to a finite verification principle.

\medskip

We also isolate a natural source of genuine arithmetic variation: fixed unit-normalized product factors give rise to multiplicative invariants \(M_Q(n)\), which may retain non-polynomial dependence on \(n\), while the remaining bounded-degree symmetric part is still forced to stabilize.

\medskip

Finally, coefficient extraction from unit-normalized products provides a systematic way to construct new admissible families. As an application, we study complete symmetric functions evaluated at punctured cosine points, where the stable-range behavior is governed by a Catalan--hypergeometric structure.

\subsection{Heat-kernel binomial formulas}

Our stable-range evaluations ultimately reduce to explicit binomial expressions for the basic cosine and sine power sums. We record here the parity-restricted binomial coefficients and the resulting heat-kernel identities that will serve as the main computational engine.

\begin{definition}\label{def:parity-binomial}
For integers \(h\ge 0\) and \(u\in\Q\), define
\[
\binom{h}{u}_{\mathrm{par}}
=
\begin{cases}
\binom{h}{u}, & \text{if \(u\in\Z\) and \(0\le u\le h\),}\\
0, & \text{otherwise.}
\end{cases}
\]
In particular, when \(u=\frac{rn+h}{2}\), the definition enforces integrality as well as the standard range condition.
\end{definition}

The next proposition (cf.\ \cite{CHJSV23}) provides the binomial-type evaluation we use throughout.

\medskip

\begin{proposition}
\label{prop:cos-powers}
Let \(n\ge 1\) and \(h\ge 0\) be integers. Define
\[
C(n,h)=\sum_{k=0}^{n-1}\cos^{h}\!\Bigl(\frac{2\pi k}{n}\Bigr),
\qquad
S(n,h)=\sum_{k=0}^{n-1}\sin^{h}\!\Bigl(\frac{2\pi k}{n}\Bigr).
\]
Then
\begin{align}
C(n,h)
&=
2^{-h}\,n
\sum_{r=-\lfloor h/n\rfloor}^{\lfloor h/n\rfloor}
\binompar{h}{\frac{rn+h}{2}},
\label{eq:Cmh-binomial}
\\[2mm]
S(n,h)
&=
2^{-h}\,n
\sum_{r=-\lfloor h/n\rfloor}^{\lfloor h/n\rfloor}
\ii^{\,rn}\,
\binompar{h}{\frac{rn+h}{2}},
\label{eq:Smh-binomial}
\end{align}
where \(\ii^2=-1\).
\end{proposition}

\section{Inverse-limit formulas and cosine-point evaluation}
\label{sec:formulas}

At each level \(n\), our evaluations involve the \((n-1)\)-tuple of punctured cosine values
\[
\alpha_{k,n}=\cos\!\Bigl(\frac{2\pi k}{n}\Bigr),
\qquad 1\le k\le n-1.
\]
Since the number of variables depends on \(n\), it is natural to organize \(n\)-dependent expressions as compatible
families of polynomials, rather than as a single polynomial in infinitely many variables. This leads to an
inverse-limit formalism.

In many natural constructions, the polynomial at level \(n\) involves the number of variables
\(m=n-1\) explicitly as a coefficient. For instance, the quadratic energy satisfies
\[
E_m(x_1,\dots,x_m)
=
m\sum_{j=1}^m x_j^2-\Bigl(\sum_{j=1}^m x_j\Bigr)^2,
\]
so the coefficient of \(\sum_j x_j^2\) is \(m\), and hence varies with the level.
To encode such level-dependent coefficients uniformly, we introduce an auxiliary indeterminate \(z\) and work over
\(\Q[z]\): we consider the family
\[
\widetilde E_m(x_1,\dots,x_m)
=
z\sum_{j=1}^m x_j^2-\Bigl(\sum_{j=1}^m x_j\Bigr)^2,
\]
which lies in \(\Q[z][x_1,\dots,x_m]\), and we recover \(E_m\) by specializing \(z=m\).
In our cyclotomic applications \(m=n-1\), so the evaluation step specializes \(z=n-1\).
Thus \(\Q[z]\) is a natural coefficient ring for identities whose coefficients vary polynomially with the level while
remaining compatible with truncation.

\subsection{Polynomial formulas}

Let
\[
R=\Q[z][x_1,x_2,x_3,\dots]
\]
be the polynomial ring in countably many variables. For each integer \(m\ge 1\), define the ideal
\[
I_m=(x_{m+1},x_{m+2},\dots)\subset R,
\]
so that \(R/I_m\cong \Q[z][x_1,\dots,x_m]\).
The quotient map induces a truncation homomorphism
\[
\phi_m:R/I_{m+1}\longrightarrow R/I_m,
\qquad x_{m+1}\longmapsto 0.
\]

\begin{definition}
\label{def:poly-formula}
A \emph{polynomial formula} is an element of the inverse limit
\[
\mathcal F_{\mathrm{poly}}
=
\varprojlim_m R/I_m.
\]
Equivalently, it is a family \(F=(F_m)_{m\ge 1}\) with
\[
F_m\in\Q[z][x_1,\dots,x_m]
\quad\text{such that}\quad
F_{m+1}(x_1,\dots,x_m,0)=F_m(x_1,\dots,x_m)
\]
for all \(m\ge 1\).
\end{definition}

Thus a polynomial formula is a level-compatible rule that assigns to each \(m\) a polynomial in \(m\) variables.

\begin{remark}
A polynomial formula need not arise from a single element of \(R\).
For example, the compatible family
\[
F_m(x_1,\dots,x_m)=\sum_{j=1}^m x_j^{\,j}
\]
defines an element of \(\mathcal F_{\mathrm{poly}}\), although no polynomial in \(R\) represents it.
\end{remark}

\subsection{Cosine-point evaluation}

Given a polynomial formula, we next define the operation that produces its level-dependent numerical values by
evaluating at punctured cosine points and specializing \(z=n-1\).

\begin{definition}
\label{def:cosine-eval}
Let \(F=(F_m)_{m\ge 1}\in\mathcal F_{\mathrm{poly}}\).
For each integer \(n\ge 2\), define its \emph{cosine-point evaluation} by
\[
a_F(n)
=
F_{n-1}\bigl(\alpha_{1,n},\dots,\alpha_{n-1,n}\bigr)\Big|_{z=n-1}.
\]
\end{definition}

This produces a numerical sequence \(a_F:\{n\ge 2\}\to\C\).

\subsection{Polynomial summability}

\begin{definition}
\label{def:poly-summable}
A polynomial formula \(F\in\mathcal F_{\mathrm{poly}}\) is said to be
\emph{eventually polynomially summable at cosine points}
if there exist an integer \(N_\ast\ge 2\) and a polynomial \(R(n)\in\Q[n]\) such that
\[
a_F(n)=R(n)
\qquad\text{for all integers }n\ge N_\ast.
\]
If one may take \(N_\ast=2\), then \(F\) is called \emph{strongly polynomially summable}.
\end{definition}

This notion is intentionally rigid: stabilization must occur to a \emph{single} polynomial function of \(n\).

\section{Universal invariants and admissible formulas}
\label{sec:admissible}
To describe stable-range behavior, it is convenient to isolate a small collection of universal invariants.
We begin by introducing the additive and multiplicative quantities that control cosine-point evaluations.

\subsection{Additive invariants}

\begin{definition}
\label{def:Ph}
For integers \(h\ge 0\) and \(n\ge 2\), define the \emph{punctured cosine power sums}
\[
P_h(n)=\sum_{k=1}^{n-1} (2\alpha_{k,n})^h .
\]
In particular, \(P_0(n)=n-1\).
\end{definition}

These invariants encode all additive symmetric information of bounded degree
in the punctured cosine configuration.

\subsection{Multiplicative invariants}

\begin{definition}
\label{def:MQ-product}
Let \(Q(z,t)\in\Q[z][[t]]\) with \(Q(z,0)=1\).
When \(Q\) is a polynomial in \(t\), define the associated multiplicative invariant
\[
M_Q(n)=\prod_{k=1}^{n-1} Q(n-1,\alpha_{k,n})
\qquad(n\ge 2).
\]
\end{definition}

Such invariants capture residual cyclotomic dependence that need not be polynomial in \(n\).

\subsection{Admissible formulas}
We now specify the class of polynomial formulas to which our rigidity results apply.
These are families with controlled algebraic complexity, consisting of a bounded-degree symmetric part together with (optional) fixed product factors.

\begin{definition}
\label{def:admissible}
Fix \(d\ge 0\).

\smallskip
\noindent
\textup{(a) Polynomial case.}
A polynomial formula \(F=(F_m)\) is called
\emph{admissible of \(x\)-degree \(\le d\)}
if for all \(m\ge 1\),
\[
F_m\in\Q[z][x_1,\dots,x_m]^{S_m},
\qquad
\deg_x(F_m)\le d.
\]

\smallskip
\noindent
\textup{(b) Product case.}
Let \(\mathbf Q=(Q_1,\dots,Q_r)\) with \(Q_i(z,0)=1\), and
\(\mathbf m=(m_1,\dots,m_r)\in\Z_{\ge0}^r\).
Given an admissible polynomial formula \(A=(A_m)\),
define
\[
F_m(x)
=
A_m(x)
\prod_{i=1}^r
\Bigl(\prod_{j=1}^m Q_i(z,x_j)\Bigr)^{m_i}.
\]
Such an \(F\) is called an
\emph{admissible product formula with bounded-degree symmetric numerator}.
\end{definition}

\section{Standard facts on symmetric polynomials}
\label{sec:symmetric-facts}

We record two classical facts that will be used repeatedly.

\begin{lemma}[Power-sum presentation in bounded degree]
\label{lem:powersum-presentation}
Fix integers \(d\ge 0\) and \(m\ge 1\).
Every symmetric polynomial
\(
G\in\Q[z][x_1,\dots,x_m]^{S_m}
\)
of total \(x\)-degree \(\le d\) can be written in the form
\[
G(x)=\Psi\bigl(p_1(x),\dots,p_d(x)\bigr),
\qquad
p_r(x)=\sum_{j=1}^m x_j^{\,r},
\]
for some polynomial \(\Psi\in\Q[z][v_1,\dots,v_d]\).
Moreover, such a presentation can be obtained algorithmically by standard elimination in the algebra of symmetric polynomials; see, e.g., \cite[Chap.~7]{StanleyEC2}.
\end{lemma}

\begin{proof}
Over \(\Q\), it is classical that for fixed \(d\), every symmetric polynomial in \(m\) variables of total degree
\(\le d\) lies in the subalgebra \(\Q[p_1,\dots,p_d]\); see \cite[Chap.~7]{StanleyEC2}.
Equivalently, the \(\Q\)-algebra homomorphism
\[
\theta_{\Q}:\ \Q[v_1,\dots,v_d]\longrightarrow \Q[x_1,\dots,x_m]^{S_m},
\qquad
v_r\longmapsto p_r(x),
\]
is surjective onto the \(\Q\)-subspace of symmetric polynomials of total degree \(\le d\).
Tensoring with \(\Q[z]\) over \(\Q\) preserves surjectivity, so the induced map
\[
\theta_{\Q[z]}:\ \Q[z][v_1,\dots,v_d]\longrightarrow \Q[z][x_1,\dots,x_m]^{S_m}
\]
is surjective onto the \(\Q[z]\)-submodule of symmetric polynomials of total degree \(\le d\).
Thus every \(G\in\Q[z][x_1,\dots,x_m]^{S_m}\) with \(\deg_x(G)\le d\) admits a representation as stated.
\end{proof}

\begin{lemma}[Stable uniqueness]
\label{lem:stable-uniqueness}
Fix integers \(d\ge 0\) and \(m\ge d\).
Suppose \(\Psi,\Psi'\in\Q[z][v_1,\dots,v_d]\) satisfy
\[
\Psi\bigl(p_1(x),\dots,p_d(x)\bigr)
=
\Psi'\bigl(p_1(x),\dots,p_d(x)\bigr)
\]
as polynomials in \(\Q[z][x_1,\dots,x_m]\).
Then \(\Psi=\Psi'\) in \(\Q[z][v_1,\dots,v_d]\).
\end{lemma}

\begin{proof}
For \(m\ge d\), the power sums \(p_1,\dots,p_d\) are algebraically independent over \(\Q\)
(see \cite[Chap.~7]{StanleyEC2}) (e.g. by Jacobian criterion at a point with distinct coordinates). Equivalently, the \(\Q\)-algebra homomorphism
\[
\theta_{\Q}:\ \Q[v_1,\dots,v_d]\longrightarrow \Q[x_1,\dots,x_m],
\qquad v_r\longmapsto p_r(x),
\]
is injective.
Since \(\Q[z]\) is a free (hence flat) \(\Q\)-module, tensoring with \(\Q[z]\) preserves injectivity, so the induced map
\[
\theta_{\Q[z]}:\ \Q[z][v_1,\dots,v_d]\longrightarrow \Q[z][x_1,\dots,x_m]
\]
is injective as well. If \(\Psi(p_1,\dots,p_d)=\Psi'(p_1,\dots,p_d)\), then
\(\theta_{\Q[z]}(\Psi-\Psi')=0\), hence \(\Psi-\Psi'=0\) by injectivity.
\end{proof}

\begin{remark}[Origin of the stable range \(n\ge d+2\)]
\label{rem:origin-stable-range}
The stable range \(n\ge d+2\) arises from two independent constraints.

\smallskip
\noindent\textbf{(1) Uniqueness of power-sum reduction.}
At cyclotomic level \(n\), the evaluation uses \(m=n-1\) variables.
To ensure that a symmetric polynomial of total \(x\)-degree \(\le d\) has a \emph{unique} expression in terms of
\(p_1,\dots,p_d\), we require \(m\ge d\), i.e.\ \(n-1\ge d\), or equivalently \(n\ge d+1\)
(Lemma~\ref{lem:stable-uniqueness}).

\smallskip
\noindent\textbf{(2) Stable formulas for \(P_h(n)\).}
Lemma~\ref{lem:Ph-stable} gives a closed formula for \(P_h(n)\) under the strict inequality \(n>h\).
To use these formulas simultaneously for all \(1\le h\le d\), it suffices to assume \(n>d\).

\smallskip
Combining (1) and (2) with a single uniform hypothesis gives
\[
n\ge d+2,
\]
which guarantees both \(m=n-1\ge d\) and \(n>d\). This is the threshold at which the rigidity argument becomes
uniformly valid.
\end{remark}

\section{Structural reduction in the stable range}
\label{sec:structural-reduction}

\begin{theorem}[Stable-range rigidity]
\label{thm:structural}
Let \(F\) be an admissible symmetric product formula of \(x\)-degree \(\le d\)
with product datum \((\mathbf Q,\mathbf m)\).
Let \(a_F(n)\) denote its cosine-point evaluation.

Set \(N_\ast=d+2\).
Then there exists a polynomial
\[
\Phi\in\Q[z][u_1,\dots,u_d]
\]
such that for all integers \(n\ge N_\ast\),
\[
a_F(n)
=
\Phi\bigl(P_1(n),\dots,P_d(n)\bigr)\Big|_{z=n-1}
\cdot
\prod_{i=1}^r M_{Q_i}(n)^{m_i}.
\]
\end{theorem}
\begin{proof}
Write
\[
F_m(x)=A_m(x)\prod_{i=1}^r
\Bigl(\prod_{j=1}^m Q_i(z,x_j)\Bigr)^{m_i},
\qquad
\deg_x(A_m)\le d.
\]
Since each $Q_i(z,0)=1$ and $F$ is truncation-compatible, we have
\[
A_{m+1}(x_1,\dots,x_m,0)=A_m(x_1,\dots,x_m)
\qquad(m\ge 1).
\]

Fix \(m\ge d\). By Lemma~\ref{lem:powersum-presentation} there exists
\(\Psi_m\in\Q[z][v_1,\dots,v_d]\) such that
\[
A_m(x)=\Psi_m\bigl(p_1(x),\dots,p_d(x)\bigr).
\]
Truncation-compatibility of \(A\) implies, as polynomials in \(\Q[z][x_1,\dots,x_m]\),
\[
\Psi_{m+1}(p_1,\dots,p_d)=\Psi_m(p_1,\dots,p_d).
\]
Since \(m\ge d\), Lemma~\ref{lem:stable-uniqueness} yields \(\Psi_{m+1}=\Psi_m\).
Hence there exists a single polynomial \(\Psi^\ast\in\Q[z][v_1,\dots,v_d]\) such that
\[
A_m(x)=\Psi^\ast(p_1(x),\dots,p_d(x))
\qquad\text{for all }m\ge d.
\]

Now let \(n\ge d+2\) and set \(m=n-1\), so \(m\ge d\). Evaluating at
\(x_k=\alpha_{k,n}\) gives, for each \(1\le h\le d\),
\[
p_h(\alpha_{1,n},\dots,\alpha_{n-1,n})
=\sum_{k=1}^{n-1}\alpha_{k,n}^h
=2^{-h}P_h(n).
\]
Define
\[
\Phi(u_1,\dots,u_d)
=
\Psi^\ast\!\Bigl(\frac{u_1}{2},\frac{u_2}{2^2},\dots,\frac{u_d}{2^d}\Bigr)
\in \Q[z][u_1,\dots,u_d].
\]
Substituting into the factorization of \(F_{n-1}\) and specializing \(z=n-1\)
yields
\[
a_F(n)
=
\Phi\bigl(P_1(n),\dots,P_d(n)\bigr)\Big|_{z=n-1}
\cdot
\prod_{i=1}^r M_{Q_i}(n)^{m_i},
\]
as claimed.
\end{proof}

\section{Stable-range evaluation of punctured cosine sums}
\label{sec:stable-Ph}
In order to make our structural reduction theorem explicit, we need closed formulas for the universal invariants
\(P_h(n)\) in the stable range.
The next lemma shows that once the cyclotomic level \(n\) exceeds the exponent \(h\), the punctured cosine power sums
simplify and become linear functions of \(n\).

\begin{lemma}
\label{lem:Ph-stable}
Let \(h\ge 0\) and \(n>h\). Then
\[
P_h(n)=
\begin{cases}
n\binom{h}{h/2}-2^h, & \text{if \(h\) is even},\\[4pt]
-2^h, & \text{if \(h\) is odd}.
\end{cases}
\]
\end{lemma}

\begin{proof}
Recall
\[
C(n,h)=\sum_{k=0}^{n-1}\cos^h\!\Bigl(\frac{2\pi k}{n}\Bigr).
\]
Since \(n>h\), we have \(\lfloor h/n\rfloor=0\), so in
Proposition~\ref{prop:cos-powers} the sum over \(r\) has only the term \(r=0\). Hence
\[
C(n,h)=2^{-h}n\binompar{h}{h/2}.
\]
In particular, \(\binompar{h}{h/2}=\binom{h}{h/2}\) if \(h\) is even and \(0\) if \(h\) is odd, so
\[
C(n,h)=2^{-h}n\binom{h}{h/2}\ \text{if \(h\) is even},\qquad
C(n,h)=0\ \text{if \(h\) is odd}.
\]

On the other hand, writing \(\alpha_{k,n}=\cos(2\pi k/n)\) we have
\[
C(n,h)=1+\sum_{k=1}^{n-1}\alpha_{k,n}^h
\quad\text{and}\quad
P_h(n)=\sum_{k=1}^{n-1}(2\alpha_{k,n})^h
=2^h\sum_{k=1}^{n-1}\alpha_{k,n}^h,
\]
so
\[
P_h(n)=2^h\bigl(C(n,h)-1\bigr).
\]
Substituting the above value of \(C(n,h)\) gives the stated formulas.
\end{proof}

\section{Consequences in the polynomial case}
\label{sec:poly-consequences}
We now specialize Theorem~\ref{thm:structural} to the purely polynomial setting, where no multiplicative factors are present.
In this case stable-range rigidity forces eventual polynomial behavior in the cyclotomic level.

\begin{corollary}[Eventual polynomiality]
\label{cor:eventual-polynomial}
Let \(F\) be an admissible symmetric polynomial formula of \(x\)-degree \(\le d\).
Then there exists a polynomial \(R_\infty(n)\in\Q[n]\) such that
\[
a_F(n)=R_\infty(n)
\qquad\text{for all integers }n\ge d+2.
\]
\end{corollary}

\begin{proof}
Apply Theorem~\ref{thm:structural} in the case with no product factors. Then there exists
\(\Phi\in\Q[z][u_1,\dots,u_d]\) such that for all \(n\ge d+2\),
\[
a_F(n)=\Phi\bigl(P_1(n),\dots,P_d(n)\bigr)\Big|_{z=n-1}.
\]
Since \(n\ge d+2\), we have \(n>d\), hence \(n>h\) for every \(1\le h\le d\). By
Lemma~\ref{lem:Ph-stable}, each \(P_h(n)\) is an affine-linear function of \(n\) on this range, and also \(z=n-1\) is affine-linear in \(n\).
Therefore the composition
\[
n\longmapsto \Phi\bigl(P_1(n),\dots,P_d(n)\bigr)\Big|_{z=n-1}
\]
is a polynomial function of \(n\). Defining \(R_\infty(n)\) to be this polynomial proves the claim.
\end{proof}

\section{Examples}
\label{sec:examples-unified}

We illustrate the stable-range rigidity principle on representative symmetric families: elementary and complete symmetric sums, a mixed cubic correlation, the quadratic energy, and two pure product formulas.  The emphasis is on how the general structure theorem yields explicit closed forms once \(n\) is large relative to the \(x\)-degree.

\begin{example}\label{ex:elem-symmetric}
Fix an integer \(r\ge 1\). Define
\[
F^{(e_r)}_m(x_1,\dots,x_m)
=
e_r(x_1,\dots,x_m)
=
\sum_{1\le i_1<\cdots<i_r\le m} x_{i_1}\cdots x_{i_r}.
\]
Then \(F^{(e_r)}\) is an admissible symmetric polynomial formula of \(x\)-degree \(\le r\).
For \(n\ge 2\) set \(\alpha_{k,n}=\cos(2\pi k/n)\) \((1\le k\le n-1)\), and define
\[
a(n)
=
F^{(e_r)}_{n-1}(\alpha_{1,n},\dots,\alpha_{n-1,n})
=
e_r(\alpha_{1,n},\dots,\alpha_{n-1,n}).
\]
Recall that the punctured cosine power sums are
\[
P_j(n)=\sum_{k=1}^{n-1}(2\alpha_{k,n})^j\qquad(j\ge 1),
\]
so in particular
\[
p_j(\alpha_{1,n},\dots,\alpha_{n-1,n})
=
\sum_{k=1}^{n-1}\alpha_{k,n}^{\,j}
=
2^{-j}P_j(n).
\]

By the determinant form of the Girard--Newton identities (\cite[Chap.~7]{StanleyEC2}), for every \(n>r\) one has
\[
e_r(\alpha_{1,n},\dots,\alpha_{n-1,n})
=
R_\infty^{(e_r)}(n),
\]
where
\[
R_\infty^{(e_r)}(n)
=
\frac{1}{r!}\,
\det\!\begin{pmatrix}
2^{-1}P_1(n) & 1      & 0      & \cdots & 0\\
2^{-2}P_2(n) & 2^{-1}P_1(n) & 2      & \cdots & 0\\
2^{-3}P_3(n) & 2^{-2}P_2(n) & 2^{-1}P_1(n) & \ddots & \vdots\\
\vdots & \vdots & \ddots & \ddots & r-1\\
2^{-r}P_r(n) & 2^{-(r-1)}P_{r-1}(n) & \cdots & 2^{-2}P_2(n) & 2^{-1}P_1(n)
\end{pmatrix}
\in \Q[n].
\]
Here the condition \(n>r\) ensures that Lemma~\ref{lem:Ph-stable} applies simultaneously to
\(P_1(n),\dots,P_r(n)\).

Equivalently, for every integer \(n>r\),
\[
\boxed{
R_\infty^{(e_r)}(n)
=
\sum_{1\le i_1<\cdots<i_r\le n-1}
\ \prod_{\nu=1}^r \cos\!\Bigl(\frac{2\pi i_\nu}{n}\Bigr)
}.
\]
\end{example}

\begin{remark}[Concrete check: \(n=8\), \(r=5\)]
In Example~\ref{ex:elem-symmetric} we are evaluating the elementary symmetric sum with strict inequalities.
For \(n=8\), the punctured cosine values are
\[
\alpha_{k,8}=\cos\!\Bigl(\frac{2\pi k}{8}\Bigr)=\cos\!\Bigl(\frac{k\pi}{4}\Bigr),
\qquad 1\le k\le 7,
\]
so the multiset \(\{\alpha_{1,8},\dots,\alpha_{7,8}\}\) equals
\[
\left\{\frac{\sqrt2}{2},\,0,\,-\frac{\sqrt2}{2},\,-1,\,-\frac{\sqrt2}{2},\,0,\,\frac{\sqrt2}{2}\right\}.
\]
Since \(\alpha_{2,8}=\alpha_{6,8}=0\), any term in
\[
e_5(\alpha_{1,8},\dots,\alpha_{7,8})
=
\sum_{1\le i_1<\cdots<i_5\le 7}\alpha_{i_1,8}\cdots\alpha_{i_5,8}
\]
that uses index \(2\) or \(6\) vanishes.  Exactly five cosine values are nonzero, so the only nonzero contribution
comes from selecting all five nonzero indices. Therefore
\[
e_5(\alpha_{1,8},\dots,\alpha_{7,8})
=
\left(\frac{\sqrt2}{2}\right)\!\left(-\frac{\sqrt2}{2}\right)\!(-1)\!\left(-\frac{\sqrt2}{2}\right)\!\left(\frac{\sqrt2}{2}\right)
=
-\frac14,
\]
i.e.
\[
\boxed{
\sum_{1\le i_1<\cdots<i_5\le 7}
\ \prod_{\nu=1}^{5}\cos\!\Bigl(\frac{2\pi i_\nu}{8}\Bigr)
=
-\frac14.
}
\]
\end{remark}

\begin{example}[A mixed cubic power sum]\label{ex:cubic-correlation}
Define the symmetric \(x\)-degree-\(3\) family
\[
F^{(2,1)}_m(x_1,\dots,x_m)
=
\sum_{\substack{1\le i,j\le m\\ i\ne j}} x_i^2x_j
=
\Bigl(\sum_{i=1}^m x_i^2\Bigr)\Bigl(\sum_{j=1}^m x_j\Bigr)-\sum_{i=1}^m x_i^3.
\]
Thus \(F^{(2,1)}\) is an admissible symmetric polynomial formula of \(x\)-degree \(\le 3\).

For \(n\ge 2\) set \(\alpha_{k,n}=\cos(2\pi k/n)\) \((1\le k\le n-1)\) and let \(m=n-1\).
For \(n\ge 4\) one has
\[
\sum_{k=1}^{n-1}\alpha_{k,n}=-1,\qquad
\sum_{k=1}^{n-1}\alpha_{k,n}^2=\frac{n-2}{2},\qquad
\sum_{k=1}^{n-1}\alpha_{k,n}^3=-1.
\]
(The restriction \(n\ge 4\) is only used for the cubic moment.)
Hence, for every integer \(n\ge 4\),
\[
\begin{aligned}
a(n)
&=
F^{(2,1)}_{n-1}(\alpha_{1,n},\dots,\alpha_{n-1,n})\\
&=
\Bigl(\sum_{k=1}^{n-1}\alpha_{k,n}^2\Bigr)\Bigl(\sum_{k=1}^{n-1}\alpha_{k,n}\Bigr)
-\sum_{k=1}^{n-1}\alpha_{k,n}^3\\
&=
\frac{n-2}{2}\cdot(-1)-(-1)
=
\frac{4-n}{2}.
\end{aligned}
\]
Therefore the (indeed exact, for \(n\ge 4\)) evaluation is
\[
\boxed{
\frac{4-n}{2}
=
\sum_{\substack{1\le i,j\le n-1\\ i\ne j}}
\cos^2\!\Bigl(\frac{2\pi i}{n}\Bigr)\,
\cos\!\Bigl(\frac{2\pi j}{n}\Bigr)
\qquad(n\ge 4).
}
\]
\end{example}

\begin{example}\label{ex:energy-quadratic-only}
Fix an integer \(m\ge 1\) and define the \emph{quadratic energy}
\[
E_m(x_1,\dots,x_m)=\sum_{1\le i<j\le m}(x_i-x_j)^2.
\]
This is a symmetric polynomial of total \(x\)-degree \(2\).
Write the power sums
\[
p_0(x_1,\dots,x_m)=\sum_{j=1}^m 1=m,\qquad
p_1(x_1,\dots,x_m)=\sum_{j=1}^m x_j,\qquad
p_2(x_1,\dots,x_m)=\sum_{j=1}^m x_j^2.
\]
Expanding \((x_i-x_j)^2=x_i^2+x_j^2-2x_ix_j\) and summing gives
\[
E_m(x)=m\,p_2(x)-p_1(x)^2=p_0(x)\,p_2(x)-p_1(x)^2.
\]
Define the symmetric family
\[
\widetilde E_m(x_1,\dots,x_m)
=
z\sum_{i=1}^m x_i^{2}
-\Bigl(\sum_{i=1}^m x_i\Bigr)^{2}.
\]
Then specializing \(z=p_0(x)=m\) recovers \(E_m\):
\[
\widetilde E_m(x)\Big|_{z=m}=E_m(x).
\]

For \(n\ge 3\) set \(m=n-1\) and \(\alpha_{k,n}=\cos(2\pi k/n)\) \((1\le k\le n-1)\).
Specializing \(x_k=\alpha_{k,n}\) yields
\[
a(n)=E_{n-1}(\alpha_{1,n},\dots,\alpha_{n-1,n})
=(n-1)\sum_{k=1}^{n-1}\alpha_{k,n}^2-\Bigl(\sum_{k=1}^{n-1}\alpha_{k,n}\Bigr)^{\!2}.
\]
The standard sums
\[
\sum_{k=1}^{n-1}\cos\!\Bigl(\frac{2\pi k}{n}\Bigr)=-1,
\qquad
\sum_{k=1}^{n-1}\cos^2\!\Bigl(\frac{2\pi k}{n}\Bigr)=\frac{n-2}{2}
\qquad(n\ge 3)
\]
give
\[
a(n)=(n-1)\cdot\frac{n-2}{2}-(-1)^2=\frac{n(n-3)}{2}.
\]
Equivalently, for every \(n\ge 3\),
\begin{equation}\label{eq:energy-cosine-boxed}
\boxed{
\frac{n(n-3)}{2}
=
\sum_{1\le i<j\le n-1}
\left(
\cos\!\Bigl(\frac{2\pi i}{n}\Bigr)
-
\cos\!\Bigl(\frac{2\pi j}{n}\Bigr)
\right)^{\!2}
}.
\end{equation}
\end{example}

\begin{example}\label{ex:pure-product_linear}
This example illustrates a \emph{pure product} formula arising from the linear factor \(1-t\).

Let
\[
Q(z,t)=1-t\in\Q[z][t],\qquad\text{so }Q(z,0)=1.
\]
Define, for each \(m\ge 1\),
\[
F_m(x_1,\dots,x_m)=\prod_{j=1}^m(1-x_j)\in\Q[z][x_1,\dots,x_m].
\]
Then \((F_m)_{m\ge 1}\) is a product formula with datum \(\mathbf Q=(Q)\) and exponent \(m_1=1\).
The associated multiplicative invariant is
\[
M_Q(n)=\prod_{k=1}^{n-1}\bigl(1-\alpha_{k,n}\bigr)
=\prod_{k=1}^{n-1}\left(1-\cos\!\Bigl(\frac{2\pi k}{n}\Bigr)\right).
\]
Let \(\omega=e^{2\pi i/n}\). Using \(2\cos(2\pi k/n)=\omega^k+\omega^{-k}\), we obtain
\[
1-\cos\!\Bigl(\frac{2\pi k}{n}\Bigr)
=
1-\frac{\omega^k+\omega^{-k}}{2}
=
\frac{(1-\omega^k)(1-\omega^{-k})}{2}.
\]
Therefore
\[
M_Q(n)
=
2^{-(n-1)}\prod_{k=1}^{n-1}(1-\omega^k)\prod_{k=1}^{n-1}(1-\omega^{-k}).
\]
Since \(k\mapsto n-k\) permutes \(\{1,\dots,n-1\}\) and \(\omega^{-k}=\omega^{\,n-k}\), we have
\[
\prod_{k=1}^{n-1}(1-\omega^{-k})=\prod_{k=1}^{n-1}(1-\omega^{k}),
\]
hence
\[
M_Q(n)=2^{-(n-1)}\Bigl(\prod_{k=1}^{n-1}(1-\omega^k)\Bigr)^{2}.
\]
Finally, from
\[
\frac{x^n-1}{x-1}=\prod_{k=1}^{n-1}(x-\omega^k)
\quad\Longrightarrow\quad
\prod_{k=1}^{n-1}(1-\omega^k)=n,
\]
we conclude that
\[
\boxed{
\prod_{k=1}^{n-1}\bigl(1-\cos(2\pi k/n)\bigr)=\frac{n^2}{2^{\,n-1}}
\qquad(n\ge 2).
}
\]
\end{example}

\begin{example}\label{ex:pure-product_quadratic}
Let
\[
Q(z,t)=1-t^2\in\Q[z][t],\qquad\text{so }Q(z,0)=1.
\]
Define the product family
\[
F_m(x_1,\dots,x_m)
=\prod_{j=1}^{m} Q(z,x_j)
=\prod_{j=1}^{m}\bigl(1-x_j^2\bigr)
\in\Q[z][x_1,\dots,x_m].
\]
This is a product formula with datum \(\mathbf Q=(Q)\) and exponent \(m_1=1\), with base numerator \(A_m\equiv 1\).
Truncation-compatibility holds since \(Q(z,0)=1\).

For \(n\ge 2\) set \(\alpha_{k,n}=\cos(2\pi k/n)\). Then the cosine-point evaluation is
\[
a(n)
=
F_{n-1}(\alpha_{1,n},\dots,\alpha_{n-1,n})\Big|_{z=n-1}
=
\prod_{k=1}^{n-1}\bigl(1-\alpha_{k,n}^2\bigr)
=
\prod_{k=1}^{n-1}\sin^2\!\Bigl(\frac{2\pi k}{n}\Bigr)
=
M_Q(n).
\]

We compute \(M_Q(n)\). Let \(\omega=e^{2\pi i/n}\). Using
\[
\sin\!\Bigl(\frac{2\pi k}{n}\Bigr)=\frac{\omega^k-\omega^{-k}}{2\ii},
\]
we obtain
\[
M_Q(n)
=
(2\ii)^{-2(n-1)}\prod_{k=1}^{n-1}(\omega^k-\omega^{-k})^2
=
4^{-(n-1)}\prod_{k=1}^{n-1}(\omega^k-\omega^{-k})^2.
\]
Now
\[
\omega^k-\omega^{-k}=\omega^{-k}(\omega^{2k}-1)=-\,\omega^{-k}(1-\omega^{2k}),
\]
so
\[
(\omega^k-\omega^{-k})^2=\omega^{-2k}(1-\omega^{2k})^2,
\]
and hence
\[
\prod_{k=1}^{n-1}(\omega^k-\omega^{-k})^2
=
\left(\prod_{k=1}^{n-1}\omega^{-2k}\right)\left(\prod_{k=1}^{n-1}(1-\omega^{2k})\right)^2.
\]
Since
\[
\prod_{k=1}^{n-1}\omega^{-2k}=\omega^{-2\cdot\frac{n(n-1)}2}=\omega^{-n(n-1)}=(\omega^n)^{-(n-1)}=1,
\]
we conclude that
\[
M_Q(n)=4^{-(n-1)}\left(\prod_{k=1}^{n-1}(1-\omega^{2k})\right)^{\!2}.
\]

If \(n\) is odd, then \(k\mapsto 2k\) permutes \(\{1,\dots,n-1\}\) modulo \(n\), hence
\[
\prod_{k=1}^{n-1}(1-\omega^{2k})=\prod_{k=1}^{n-1}(1-\omega^k)=n.
\]
If \(n\) is even, then \(k=n/2\) gives \(\omega^{2k}=\omega^n=1\), so one factor \(1-\omega^{2k}\) vanishes and
\(\prod_{k=1}^{n-1}(1-\omega^{2k})=0\).
Thus
\[
\boxed{
\prod_{k=1}^{n-1}\left(1-\cos^2\!\Bigl(\frac{2\pi k}{n}\Bigr)\right)
=
\prod_{k=1}^{n-1}\sin^2\!\Bigl(\frac{2\pi k}{n}\Bigr)
=
\begin{cases}
\dfrac{n^2}{4^{\,n-1}}, & \text{if \(n\) is odd},\\[8pt]
0, & \text{if \(n\) is even},
\end{cases}
\qquad(n\ge 2).
}
\]
\end{example}

\section{Coefficient extraction from fixed products}
\label{sec:operations-fixed-products}

In several applications one encounters expressions that are not themselves admissible polynomial formulas,
but arise from admissible product-type families by \emph{coefficient extraction} in an auxiliary parameter.
The next lemma records a formal mechanism that produces bounded-degree symmetric coefficient families from a
unit-normalized fixed product, so that the bounded-degree summability framework applies.

\begin{lemma}\label{lem:coeff-extraction-product}
Let
\[
Q(z,t)=\sum_{\ell\ge 0} q_\ell(z)\,t^\ell\ \in\ \Q[z][[t]]
\]
be a formal power series in \(t\) with coefficients in \(\Q[z]\), normalized by \(q_0(z)=Q(z,0)=1\).
Fix \(m\ge 1\). Consider the formal power series in an auxiliary variable \(s\)
\[
B_m(s;x_1,\dots,x_m)
=
\prod_{j=1}^m Q\bigl(z,\,s x_j\bigr)
\ \in\ \Q[z][x_1,\dots,x_m][[s]].
\]
Write
\[
B_m(s;x)=\sum_{r\ge 0} b_{m,r}(x)\,s^r,
\qquad b_{m,r}(x)\in \Q[z][x_1,\dots,x_m]^{S_m}.
\]
Then, for each fixed \(r\ge 0\), the family \(b_r=(b_{m,r})_{m\ge 1}\) defines a symmetric polynomial formula
in \(\mathcal F_{\mathrm{poly}}\). Moreover, for each fixed \(r\) one has the uniform degree bound
\[
\deg_x\bigl(b_{m,r}\bigr)\le r\qquad\text{for all }m\ge 1.
\]
In particular, for each fixed \(r\), the cosine-point evaluation
\[
a_{b_r}(n)=b_{n-1,r}\bigl(\alpha_{1,n},\dots,\alpha_{n-1,n}\bigr)\Big|_{z=n-1}
\]
is eventually polynomial in \(n\) (indeed for all \(n\ge r+2\)).
\end{lemma}

\begin{proof}
\emph{Existence of coefficients and symmetry.}
For each \(j\),
\[
Q(z,sx_j)=\sum_{\ell\ge 0} q_\ell(z)\,s^\ell x_j^\ell \in \Q[z][x_j][[s]].
\]
Hence
\[
B_m(s;x)=\prod_{j=1}^m Q(z,sx_j)\in \Q[z][x_1,\dots,x_m][[s]]
\]
is well-defined, and writing
\(
B_m(s;x)=\sum_{r\ge 0} b_{m,r}(x)\,s^r
\)
defines unique coefficients \(b_{m,r}(x)\in \Q[z][x_1,\dots,x_m]\).
Since \(B_m\) is symmetric in \(x_1,\dots,x_m\), each coefficient \(b_{m,r}\) lies in
\(\Q[z][x_1,\dots,x_m]^{S_m}\).

\medskip

\emph{Polynomiality and degree bound.}
Expanding the product, the coefficient of \(s^r\) is
\[
b_{m,r}(x)=\sum_{\substack{\ell_1,\dots,\ell_m\ge 0\\ \ell_1+\cdots+\ell_m=r}}
\Bigl(\prod_{j=1}^m q_{\ell_j}(z)\Bigr)\,x_1^{\ell_1}\cdots x_m^{\ell_m}.
\]
The indexing set is finite, so \(b_{m,r}(x)\) is a finite \(\Q[z]\)-linear combination of monomials.
Each monomial has total \(x\)-degree \(\ell_1+\cdots+\ell_m=r\), hence \(\deg_x(b_{m,r})\le r\).

\medskip

\emph{Truncation-compatibility.}
Using \(Q(z,0)=1\), we have in \(\Q[z][x_1,\dots,x_m][[s]]\),
\[
B_{m+1}(s;x_1,\dots,x_m,0)
=
\prod_{j=1}^{m} Q(z,sx_j)\cdot Q(z,0)
=
B_m(s;x_1,\dots,x_m).
\]
Comparing coefficients of \(s^r\) gives
\(
b_{m+1,r}(x_1,\dots,x_m,0)=b_{m,r}(x_1,\dots,x_m)
\),
so \((b_{m,r})_{m\ge 1}\) defines an element of \(\mathcal F_{\mathrm{poly}}\).

\medskip

\emph{Eventual polynomiality.}
For fixed \(r\), the family \(b_r\) is an admissible symmetric polynomial formula of \(x\)-degree \(\le r\),
so Corollary~\ref{cor:eventual-polynomial} applies with \(d=r\), giving eventual polynomiality for \(n\ge r+2\).
\end{proof}

\subsection{Application: complete symmetric sums at cosine points}
\label{subsec:application-sec9-complete}

The goal of this subsection is to illustrate how the general summability mechanism applies to complete symmetric functions
evaluated at punctured cosine points.  The key observation is that the complete symmetric polynomial \(h_r\) is obtained by
coefficient extraction from the unit-normalized formal power series \(Q(t)=(1-t)^{-1}\in\Q[[t]]\) (independent of \(z\)).

\medskip

Fix an integer \(r\ge 1\).  For each \(m\ge 1\) define the \emph{complete symmetric polynomial}
\[
h_r(x_1,\dots,x_m)
=
\sum_{1\le i_1\le\cdots\le i_r\le m} x_{i_1}\cdots x_{i_r},
\]
and set
\[
F^{(h_r)}_m(x_1,\dots,x_m)=h_r(x_1,\dots,x_m).
\]
Then \(F^{(h_r)}=(F^{(h_r)}_m)_{m\ge 1}\) defines a symmetric polynomial formula in \(\mathcal F_{\mathrm{poly}}\), and it has
uniform \(x\)-degree \(\deg_x(F^{(h_r)}_m)=r\).  Indeed, by the classical generating identity
\cite[Chap.~7]{StanleyEC2},
\[
\prod_{j=1}^m \frac{1}{1-tx_j}
=
\sum_{q\ge 0} h_q(x_1,\dots,x_m)\,t^q,
\]
the polynomial \(h_r\) is obtained by coefficient extraction from \(Q(t)=(1-t)^{-1}\).  Truncation-compatibility and the
uniform degree bound therefore follow from Lemma~\ref{lem:coeff-extraction-product}.

\medskip

For \(n\ge 2\) set \(m=n-1\) and recall the punctured cosine points
\[
\alpha_{k,n}=\cos\!\Bigl(\frac{2\pi k}{n}\Bigr),
\qquad 1\le k\le n-1.
\]
We study the cosine evaluation
\[
a_r(n)=h_r(\alpha_{1,n},\dots,\alpha_{n-1,n}).
\]
By Corollary~\ref{cor:eventual-polynomial}, there is a stable range \(n\ge r+2\) in which \(a_r(n)\) agrees with a polynomial
\(R_\infty^{(h_r)}(n)\in\Q[n]\).  Our goal is to compute this stable-range polynomial explicitly.

\bigskip

\subsubsection{The Catalan generating series}

We begin by introducing the auxiliary series
\begin{equation}\label{eq:def-A}
A(t)=\frac{2}{1+\sqrt{1-t^2}}\in\Q[[t]].
\end{equation}
Rationalizing the denominator gives
\[
A(t)=\frac{2(1-\sqrt{1-t^2})}{t^2}.
\]
Setting \(u=t^2/4\), this becomes the \emph{Catalan generating function \(C(u)\):}
\[
A(t)=\frac{1-\sqrt{1-4u}}{2u}=C(u).
\]
Hence
\[
A(t)=\sum_{\ell\ge 0}\Catnum{\ell}\left(\frac{t^2}{4}\right)^{\!\ell}
=\sum_{\ell\ge 0} c_\ell\,t^{2\ell},
\qquad
c_\ell=\frac{\Catnum{\ell}}{4^\ell},
\qquad
\Catnum{\ell}=\frac{1}{\ell+1}\binom{2\ell}{\ell}.
\]

\begin{lemma}\label{lem:A-identity}
The series \(A(t)=\dfrac{2}{1+\sqrt{1-t^2}}\in\Q[[t]]\) satisfies
\[
A(t)=1+\frac{t^2}{4}\,A(t)^2.
\]
\end{lemma}

\begin{proof}
This is the Catalan functional equation in disguise (see \cite[Chap.~7]{StanleyEC2}): if one sets \(u=t^2/4\) and writes
\(C(u)=\dfrac{1-\sqrt{1-4u}}{2u}\), then \(C(u)=1+uC(u)^2\) and \(A(t)=C(t^2/4)\).
\end{proof}

\begin{lemma}\label{lem:logA-series}
Let \(A(t)\) be as in \eqref{eq:def-A}. Then in \(\Q[[t]]\) one has
\[
\log A(t)=\sum_{j\ge 1}\frac{1}{2j}\binom{2j}{j}\left(\frac{t^2}{4}\right)^{\!j}.
\]
\end{lemma}

\begin{proof}
Set \(u=t^2/4\), so \(1-t^2=1-4u\), and write
\[
A(t)=\frac{2}{1+\sqrt{1-4u}}.
\]
Define \(B(u)=\log\!\left(\frac{2}{1+\sqrt{1-4u}}\right)\in\Q[[u]]\), so that \(\log A(t)=B(t^2/4)\).  Let
\(s(u)=\sqrt{1-4u}\), so \(s'(u)=-2/s(u)\). Then
\[
B(u)=\log 2-\log(1+s(u)),
\qquad
B'(u)=-\frac{s'(u)}{1+s(u)}=\frac{2}{s(u)(1+s(u))}.
\]
Using \(1-s(u)^2=4u\), one checks
\[
\frac{2}{s(1+s)}=\frac{1}{2u}\left(\frac{1}{s}-1\right),
\]
hence
\[
B'(u)=\frac{1}{2u}\left(\frac{1}{\sqrt{1-4u}}-1\right).
\]
The binomial series gives
\[
\frac{1}{\sqrt{1-4u}}=\sum_{j\ge 0}\binom{2j}{j}u^j,
\]
so
\[
B'(u)=\frac{1}{2u}\sum_{j\ge 1}\binom{2j}{j}u^j
=\frac12\sum_{j\ge 1}\binom{2j}{j}u^{j-1}.
\]
Integrating term-by-term in \(\Q[[u]]\) and using \(B(0)=0\) yields
\[
B(u)=\sum_{j\ge 1}\frac{1}{2j}\binom{2j}{j}u^j.
\]
Substituting \(u=t^2/4\) gives the desired identity.
\end{proof}

For each integer \(n\) write
\begin{equation}\label{eq:def-aell}
A(t)^n=\sum_{\ell\ge 0} a_\ell(n)\,t^{2\ell}.
\end{equation}

\begin{lemma}\label{lem:am-recurrence}
For all \(\ell\ge 1\) and all integers \(n\) one has
\[
a_\ell(n)=a_\ell(n-1)+\frac14\,a_{\ell-1}(n+1),
\qquad
a_0(n)=1.
\]
\end{lemma}

\begin{proof}
Multiply the identity of Lemma~\ref{lem:A-identity} by \(A(t)^{n-1}\) to obtain
\[
A(t)^n=A(t)^{n-1}+\frac{t^2}{4}A(t)^{n+1}.
\]
Comparing coefficients of \(t^{2\ell}\) yields the recurrence.
\end{proof}

\begin{lemma}\label{lem:am-closed}
For every \(\ell\ge 0\) and every integer \(n\),
\[
a_\ell(n)
=
\frac{1}{4^\ell}\cdot \frac{n}{n+2\ell}\binom{n+2\ell}{\ell}.
\]
Equivalently, for \(\ell\ge 1\),
\[
a_\ell(n)=\frac{n}{2^{2\ell}\,\ell!}\prod_{j=\ell+1}^{2\ell-1}(n+j),
\qquad a_0(n)=1.
\]
\end{lemma}

\begin{proof}
Define
\[
b_\ell(n)=\frac{1}{4^\ell}\cdot \frac{n}{n+2\ell}\binom{n+2\ell}{\ell},
\qquad (\ell\ge 0).
\]
Then \(b_0(n)=1\).  For \(\ell\ge 1\), the stated product form is obtained by expanding the binomial coefficient and
cancelling the factor \(n+2\ell\).

\medskip

\noindent\textbf{Claim.} For every \(\ell\ge 1\) and every integer \(n\),
\[
b_\ell(n)=b_\ell(n-1)+\frac14\,b_{\ell-1}(n+1).
\]
Granting the claim, Lemma~\ref{lem:am-recurrence} and the identity \(a_0(n)=b_0(n)=1\) imply \(a_\ell(n)=b_\ell(n)\)
for all \(\ell\ge 0\), since the coefficients in \eqref{eq:def-aell} are uniquely determined by comparing the coefficient
of \(t^{2\ell}\) in the identity
\(
A(t)^n=A(t)^{n-1}+\frac{t^2}{4}A(t)^{n+1}.
\)

\medskip

\noindent\textbf{Verification of the claim.}
Fix \(\ell\ge 1\).  Write
\[
b_\ell(n)=\frac{n}{2^{2\ell}\,\ell!}\,P_\ell(n),
\qquad
P_\ell(n)=\prod_{j=\ell+1}^{2\ell-1}(n+j),
\]
and set \(S_\ell(n)=\prod_{j=\ell+1}^{2\ell-2}(n+j)\) (so \(S_\ell(n)=1\) if \(\ell=1\)).
Then
\[
P_\ell(n)=(n+2\ell-1)S_\ell(n),\quad
P_\ell(n-1)=(n+\ell)S_\ell(n),\quad
P_{\ell-1}(n+1)=S_\ell(n).
\]
Substituting into the recurrence reduces it (after cancelling the common factor
\(S_\ell(n)/(2^{2\ell}\ell!)\)) to
\[
n(n+2\ell-1)=(n-1)(n+\ell)+\ell(n+1),
\]
which holds by direct expansion.
\end{proof}

\begin{remark}\label{rem:am-hypergeometric}
The coefficients \(a_\ell(n)\) admit the hypergeometric representation
\[
A(t)^n
=
{}_2F_{1}\!\left(\frac n2,\frac{n+1}{2};\,n+1;\,t^2\right),
\qquad
a_\ell(n)=
\frac{\left(\frac n2\right)_\ell\left(\frac{n+1}{2}\right)_\ell}{(n+1)_\ell\,\ell!},
\]
where \((x)_\ell\) denotes the Pochhammer symbol; see \cite[\S15.2]{DLMF}.
\end{remark}

\subsubsection*{Generating function and stable truncation}

\begin{lemma}\label{lem:hr-genfn}
For every integer \(n\ge 2\) one has the formal identity
\[
\sum_{r\ge 0} h_r(\alpha_{1,n},\dots,\alpha_{n-1,n})\,t^r
=
\prod_{k=1}^{n-1}\frac{1}{1-\alpha_{k,n}t}.
\]
\end{lemma}

\begin{proof}
For indeterminates \(x_1,\dots,x_{n-1}\) one has the classical identity
\[
\sum_{r\ge 0} h_r(x_1,\dots,x_{n-1})\,t^r
=
\prod_{k=1}^{n-1}\frac{1}{1-x_k t},
\]
see for instance \cite[Chap.~7]{StanleyEC2}.  Specialize \(x_k=\alpha_{k,n}\).
\end{proof}

\begin{lemma}\label{lem:trunk}
Fix an integer \(R\ge 1\) and assume \(n>R\).  Let
\[
H_n(t)=\sum_{r\ge 0} h_r(\alpha_{1,n},\dots,\alpha_{n-1,n})\,t^r
=\prod_{k=1}^{n-1}\frac{1}{1-\alpha_{k,n}t}\ \in\ \C[[t]].
\]
Then modulo \(t^{R+1}\) one has the congruence
\[
H_n(t)
\equiv
(1-t)\,A(t)^{\,n}
\pmod{t^{R+1}},
\]
where \(A(t)=\dfrac{2}{1+\sqrt{1-t^2}}\in\Q[[t]]\).
Equivalently, the image of \(H_n(t)\) in \(\C[[t]]/(t^{R+1})\) coincides with the image of
\((1-t)A(t)^n\in\Q[[t]]\).
\end{lemma}

\begin{proof}
By Lemma~\ref{lem:hr-genfn} we have in \(\C[[t]]\) the identity
\[
H_n(t)
=
\prod_{k=1}^{n-1}\frac{1}{1-\alpha_{k,n}t}.
\]
Since \(H_n(0)=1\), we may take logarithms in \(\C[[t]]\) and obtain
\[
\log H_n(t)
=
-\sum_{k=1}^{n-1}\log(1-\alpha_{k,n}t)
=
\sum_{j\ge 1}\frac{t^j}{j}\sum_{k=1}^{n-1}\alpha_{k,n}^{\,j}.
\]
Reducing modulo \(t^{R+1}\), only the terms \(1\le j\le R\) contribute:
\[
\log H_n(t)\equiv
\sum_{j=1}^{R}\frac{t^j}{j}\,C_j^{*}(n)
\pmod{t^{R+1}},
\qquad
C_j^{*}(n)=\sum_{k=1}^{n-1}\alpha_{k,n}^{\,j}.
\]

Assume \(n>R\), so \(n>j\) for every \(1\le j\le R\).
Write
\[
C(n,j)=\sum_{k=0}^{n-1}\cos^j\!\Bigl(\frac{2\pi k}{n}\Bigr),
\qquad
C_j^{*}(n)=\sum_{k=1}^{n-1}\cos^j\!\Bigl(\frac{2\pi k}{n}\Bigr)=C(n,j)-1.
\]
By Proposition~\ref{prop:cos-powers}, when \(n>j\) the sum in \eqref{eq:Cmh-binomial}
has only the term \(r=0\), hence
\[
C(n,j)=2^{-j}\,n\,\binompar{j}{j/2}.
\]
Therefore:
\begin{itemize}
\item if \(j\) is odd then \(\binompar{j}{j/2}=0\), so \(C(n,j)=0\) and hence \(C_j^{*}(n)=-1\);
\item if \(j=2m\) is even then \(\binompar{2m}{m}=\binom{2m}{m}\), so
\[
C_{2m}^{*}(n)=2^{-2m}n\binom{2m}{m}-1.
\]
\end{itemize}

Substituting into \(\log H_n(t)\) gives, modulo \(t^{R+1}\),
\[
\log H_n(t)\equiv
-\sum_{j=1}^{R}\frac{t^j}{j}
\;+\;
n\sum_{1\le m\le \lfloor R/2\rfloor}\frac{t^{2m}}{2m}\frac{1}{4^m}\binom{2m}{m}
\pmod{t^{R+1}}.
\]
The first term satisfies
\[
-\sum_{j=1}^{R}\frac{t^j}{j}\equiv \log(1-t)\pmod{t^{R+1}}
\]
in \(\Q[[t]]\).
For the even part, Lemma~\ref{lem:logA-series} implies
\[
\log A(t)=\sum_{m\ge 1}\frac{1}{2m}\binom{2m}{m}\left(\frac{t^2}{4}\right)^{m},
\]
so truncating modulo \(t^{R+1}\) yields
\[
\sum_{1\le m\le \lfloor R/2\rfloor}\frac{t^{2m}}{2m}\frac{1}{4^m}\binom{2m}{m}
\equiv \log A(t)\pmod{t^{R+1}}.
\]
Hence
\[
\log H_n(t)\equiv \log(1-t)+n\log A(t)\pmod{t^{R+1}}.
\]

Both sides lie in \(t\,\C[[t]]\), and the exponential map
\(\exp:t\,\C[[t]]\to 1+t\,\C[[t]]\) preserves congruences modulo \(t^{R+1}\).
Exponentiating yields
\[
H_n(t)\equiv (1-t)\,A(t)^{\,n}\pmod{t^{R+1}},
\]
as claimed.
\end{proof}

\subsubsection*{Stable-range evaluation of \(h_r(\alpha_{1,n},\dots,\alpha_{n-1,n})\)}

We now compute the stable-range polynomial \(R_\infty^{(h_r)}(n)\) by extracting the coefficient of \(t^r\) from \((1-t)A(t)^n\).

\begin{theorem}\label{thm:hr-pattern}
Let \(r\ge 2\) and set \(N_\ast=r+2\).
Write \(r=2m\) or \(r=2m+1\) with \(m\ge 1\).
Then for all integers \(n\ge N_\ast\),
\[
h_r(\alpha_{1,n},\dots,\alpha_{n-1,n})
=
\begin{cases}
\displaystyle \frac{n}{2^{2m}m!}\prod_{j=m+1}^{2m-1}(n+j), & \text{if \(r=2m\) is even},\\[10pt]
\displaystyle -\frac{n}{2^{2m}m!}\prod_{j=m+1}^{2m-1}(n+j), & \text{if \(r=2m+1\) is odd}.
\end{cases}
\]
Equivalently, these agree with polynomials in \(n\) for all \(n\ge r+2\).
\end{theorem}

\begin{proof}
Take \(R=r\) in Lemma~\ref{lem:trunk} and assume \(n\ge r+2\), so in particular \(n>r\) and the congruence holds modulo \(t^{r+1}\):
\[
H_n(t)\equiv (1-t)\,A(t)^n \pmod{t^{r+1}}.
\]
Writing
\(
H_n(t)=\sum_{q\ge 0} h_q(\alpha_{1,n},\dots,\alpha_{n-1,n})\,t^q,
\)
the coefficient of \(t^r\) agrees on both sides:
\[
h_r(\alpha_{1,n},\dots,\alpha_{n-1,n})
=
\Coeff_{t^r}\!\bigl((1-t)A(t)^n\bigr).
\]

From \eqref{eq:def-aell} we have
\[
A(t)^n=\sum_{\ell\ge 0} a_\ell(n)\,t^{2\ell},
\]
so
\[
(1-t)A(t)^n
=
\sum_{\ell\ge 0} a_\ell(n)\,t^{2\ell}
-
\sum_{\ell\ge 0} a_\ell(n)\,t^{2\ell+1}.
\]
Hence the coefficient of \(t^r\) is \(a_m(n)\) if \(r=2m\) is even and \(-a_m(n)\) if \(r=2m+1\) is odd.
Applying Lemma~\ref{lem:am-closed} completes the proof.
\end{proof}

\subsubsection*{Examples}

\noindent\textbf{Case \(r=6\).}
For every integer \(n\ge 8\),
\begin{equation}\label{eq:hr6-explicit}
\boxed{
\sum_{1\le i_1\le\cdots\le i_6\le n-1}
\ \prod_{\nu=1}^{6}\cos\!\Bigl(\frac{2\pi i_\nu}{n}\Bigr)
=
\frac{n(n+4)(n+5)}{384}.
}
\end{equation}

\medskip

\noindent\textbf{Case \(r=7\).}
For every integer \(n\ge 9\),
\begin{equation}\label{eq:hr7-explicit}
\boxed{
\sum_{1\le i_1\le\cdots\le i_7\le n-1}
\ \prod_{\nu=1}^{7}\cos\!\Bigl(\frac{2\pi i_\nu}{n}\Bigr)
=
-\frac{n(n+4)(n+5)}{384}.
}
\end{equation}

\medskip

For instance, \eqref{eq:hr7-explicit} gives at \(n=9\) the exact value
\[
-\frac{9(9+4)(9+5)}{384}
=
-\frac{273}{64}
=
-4.265625.
\]
Brute-force coefficient extraction from \(\prod_{k=1}^{8}(1-\alpha_{k,9}t)^{-1}\), truncated at \(t^7\), returns the same value.

\medskip

While the general theory developed above guarantees only \emph{stable-range} polynomial behavior (for each fixed \(r\), once
\(n\) is large compared to \(r\)), the coefficient-extraction viewpoint sometimes yields more: in favorable cyclotomic
situations one can compute, for each fixed level \(n\), the \emph{entire} generating series
\(\sum_{r\ge 0}h_r(\alpha_{1,n},\dots,\alpha_{n-1,n})\,s^r\) in closed form.

\medskip

Let \(T_n(x)\) be the Chebyshev polynomial of the first kind, characterized by \(T_n(\cos\theta)=\cos(n\theta)\).
A standard factorization is
\[
T_n(x)-1
=
2^{n-1}(x-1)\prod_{k=1}^{n-1}\bigl(x-\cos(2\pi k/n)\bigr)
\]
(see, e.g., \cite[Ch.~2]{Rivlin1990}).
Substituting \(x=1/s\) and rearranging gives the exact product identity
\[
\prod_{k=1}^{n-1}\bigl(1-s\alpha_{k,n}\bigr)
=
\frac{s^{\,n}\bigl(T_n(1/s)-1\bigr)}{2^{n-1}(1-s)}.
\]
Therefore,
\[
\boxed{
\sum_{r\ge 0} h_r(\alpha_{1,n},\dots,\alpha_{n-1,n})\,s^r
=
\prod_{k=1}^{n-1}\frac{1}{1-s\alpha_{k,n}}
=
\frac{2^{n-1}(1-s)}{s^{\,n}\bigl(T_n(1/s)-1\bigr)}.
}
\]

\medskip

We note that this identity is \emph{global} in the coefficient parameter \(r\) for each fixed level \(n\).
Extracting the coefficient of \(s^r\) and then passing to the stable range \(n\ge r+2\) recovers the eventual polynomiality
(and the explicit stable-range values) predicted by Theorem~\ref{thm:hr-pattern}.

\section{Conclusions}
\label{sec:conclusions}

Even though our proofs are algebraic and combinatorial, the reason stabilization occurs is quite down to earth.
The basic step is to rewrite each cosine in terms of a root of unity:
\[
\cos\!\Bigl(\frac{2\pi k}{n}\Bigr)=\frac{\omega_n^k+\omega_n^{-k}}{2},
\qquad \omega_n=e^{2\pi i/n}.
\]
Once you do this, every expression built from the punctured cosine points becomes a sum of terms involving powers of \(\omega_n^k\).
When you then sum over \(k=1,\dots,n-1\), the only sums that appear are geometric sums of the form
\(\sum_{k=1}^{n-1}\omega_n^{mk}\).
Crucially, if your original expression is symmetric and has bounded total \(x\)-degree \(\le d\), then there are only finitely many integers \(m\)
that can show up in this way: the degree bound limits how complicated the exponents can be.
So once \(n\) is larger than that bound, there is simply no room for new cyclotomic behavior to enter.
In the stable range, the evaluation is forced to depend only on finitely many universal pieces of data.

\smallskip

The inverse-limit formalism is just a convenient way to state this cleanly.
It lets us treat ``one polynomial at each level'' as a single truncation-compatible object.
With that packaging in place, bounded-degree symmetry gives a uniform reduction:
for \(n\ge d+2\), the cosine-point evaluation factors through the finitely many punctured power sums
\(P_1(n),\dots,P_d(n)\).
In the purely polynomial case, this immediately implies that the evaluation eventually agrees with a single polynomial in \(n\),
and therefore any identity valid for all \(n\) can be checked at only finitely many small levels.

The product extension clarifies what can and cannot stabilize.
The bounded-degree symmetric polynomial part is the piece that must collapse to the finite list \(P_1,\dots,P_d\).
Any genuine cyclotomic variation that survives is confined to the explicit multiplicative factors \(M_Q(n)\),
which are visible and computable on their own.

Finally, the same viewpoint suggests several natural variants.

\medskip

\noindent\textbf{(1) Replacing full symmetry by Galois invariance.}
In this article we impose full \(S_{n-1}\)-symmetry in the variables \((x_1,\dots,x_{n-1})\).
A more arithmetic (and more natural) requirement is invariance under the cyclotomic Galois group
\[
G_n=\Gal(\Q(\omega_n)/\Q)\cong (\Z/n\Z)^\times,
\]
acting on the punctured cosine points by
\(\alpha_{k,n}\mapsto \alpha_{ak,n}\) for \(a\in(\Z/n\Z)^\times\).
This is the minimal condition one should expect if one wants \emph{rational} evaluations:
indeed, the numbers \(\alpha_{k,n}\) lie in the maximal real subfield \(\Q(\omega_n+\omega_n^{-1})\), and a value obtained from them
will in general live in that field.  If an evaluation is to land in \(\Q\), it must be fixed by every automorphism of that field,
hence it must be invariant under the induced \(G_n\)-action.

From the structural point of view, bounded-degree \(G_n\)-invariant families should again admit a stable-range reduction,
but now through finitely many \emph{orbit-sum} invariants: instead of symmetric power sums over all indices,
one expects traces along \(G_n\)-orbits (or, more generally, traces from intermediate subfields)
to replace the role of \(P_1(n),\dots,P_d(n)\).
This would lead to an analogue of the rigidity theorem in which the stable input is a finite list of trace-type cyclotomic invariants.

\medskip

\noindent\textbf{(2) Twisted evaluations by Dirichlet characters.}
A different direction is to insert arithmetic weights (as in \cite{CHJSV23}) , for example
\[
\sum_{k=1}^{n-1}\chi(k)\,F_{n-1}\!\Bigl(\cos\!\Bigl(\frac{2\pi k}{n}\Bigr)\Bigr),
\]
where \(\chi\) is a Dirichlet character modulo \(n\) (or modulo a divisor of \(n\)).
Here the symmetry is no longer permutation symmetry of the indices, but rather the harmonic-analytic structure imposed by \(\chi\).
After rewriting cosines in terms of \(\omega_n^k\), the basic building blocks become twisted geometric sums
\[
\sum_{k=1}^{n-1}\chi(k)\,\omega_n^{mk},
\]
and the relevant ``finite set of inputs'' should therefore be a finite list of such character-weighted cyclotomic sums,
i.e.\ Gauss--Jacobi/Ramanujan-type quantities, replacing the untwisted invariants \(P_h(n)\).
In favorable cases one might again expect a stable-range rigidity statement: bounded-degree constructions would force the twisted evaluation
to factor through finitely many explicitly describable twisted invariants, together with any residual product-type contributions.

\section*{Acknowledgments}
Carlos A. Cadavid gratefully acknowledges the financial support of Universidad EAFIT (Colombia) for the project Study and
Applications of Diffusion Processes of Importance in Health and Computation (project code 11740052022).
Juan D. Vélez gratefully acknowledges the Universidad Nacional de Colombia for its support during this research.

\bibliographystyle{plain}

\section*{Author information}

Carlos A.\ Cadavid\\
Department of Mathematics, Universidad Eafit, Medell\'in, Colombia\\
\texttt{ccadavid@eafit.edu.co}

\medskip
Juan D.\ V\'elez\\
Department of Mathematics, Universidad Nacional de Colombia, Medell\'in, Colombia\\
\texttt{jdvelez@unal.edu.co}

\end{document}